\documentclass[a4paper]{siamltex}
\usepackage{amssymb,amsmath}
\usepackage{url}
\usepackage{xspace}
\usepackage[colorlinks,bookmarks=true,bookmarksopen=true,linkcolor=blue,citecolor=blue]{hyperref}
\usepackage[textsize=small]{todonotes}
\usepackage[ulem=normalem,draft]{changes}

\newtheorem{remark}[theorem]{{\it Remark }\rm }

\newcommand{\tcr}{}

\newcommand{\dx}{\,\mathrm{d}x}
\newcommand{\dt}{\,\mathrm{d}t}

\newcommand{\Pb}{\eqref{eq_Pp}\xspace}
\newcommand{\PbM}{\eqref{eq_PbM}\xspace}
\newcommand{\PbMk}{\mbox{\rm (P$_{M_k}$)}\xspace}

\newcommand{\Pbn}{\eqref{eq_Pn}\xspace}
\newcommand{\PbMn}{\eqref{eq_PbMN}\xspace}

\newcommand{\ds}{\,\mathrm{d}s}
\newcommand{\e}{{\rm e}}

\def\R{\mathbb R}

\title{{A Note on Existence of Solutions to Control Problems of Semilinear Partial Differential Equations}
\thanks{The first author was supported by MCIN/ AEI/10.13039/501100011033/ under research project PID2020-114837GB-I00.
The second author was partially supported by the German Research Foundation DFG under project grant Wa 3626/3-2.
}}

\author{Eduardo Casas\thanks{Departmento de Matem\'{a}tica Aplicada y Ciencias de la Computaci\'{o}n, E.T.S.I. Industriales y de Telecomunicaci\'on, Universidad de Cantabria, 39005 Santander, Spain, {\tt eduardo.casas@unican.es}.}
\and Daniel Wachsmuth\thanks{Institut f\"ur Mathematik, Universit\"at W\"urzburg, 97074 W\"urzburg, Germany, {\tt  daniel.wachsmuth@mathematik.uni-wuerzburg.de}.}
}

\begin{document}
\maketitle

\begin{abstract}
In this paper, we study optimal control problems of semilinear elliptic and parabolic equations.
A tracking cost functional, quadratic in the control and state variables, is considered.
No control constraints are imposed. We prove that the corresponding state equations are well-posed for controls in $L^2$.
However, it is well-known that in the $L^2$ framework the mappings involved in the control problem are not
Frechet differentiable in general, which makes any analysis of the optimality conditions challenging.
Nevertheless, we prove that every $L^2$ optimal control belongs to $L^\infty$,
and consequently standard optimality conditions are available.
\end{abstract}



\section{Introduction}
\label{S1}
In this paper, we study the following optimal control problem
\begin{equation} \label{eq_Pp} \tag{P}
 \inf_{u \in L^2(Q)} J(u):= \frac{1}{2}\int_Q [(y_u - y_d)^2 + \alpha u^2]\dx\dt,
\end{equation}
where $\alpha > 0$ and $y_u$ is the solution of the semilinear parabolic equation
\begin{equation}
\left\{\begin{array}{l}\displaystyle\frac{\partial y}{\partial t} + Ay + f(y) =  u \ \mbox{ in } Q = \Omega \times (0,T),\\y  =  0\ \mbox{ on } \Sigma = \Gamma \times (0,T),\ \ y(x,0) = y_0(x)\ \text{ in } \Omega.\end{array}\right.
\label{E1.1}
\end{equation}
Here, $A$ denotes an elliptic operator in the bounded domain $\Omega \subset \mathbb{R}^n$, $n\ge 2$, whose boundary is denoted by $\Gamma$, $T \in (0,\infty)$ is fixed, $y_0 \in L^\infty(\Omega)$, and $f:\mathbb{R} \longrightarrow \mathbb{R}$ is a given function.
Additionally, we assume that $y_d \in L^p(0,T;L^q(\Omega))$ with $p, q \in [2,\infty]$ and $\frac{1}{p} + \frac{n}{2q} < 1$ is a given function.
Assumptions on the nonlinear term $f$ in the state equation will be established later.
Let us emphasize that we do not impose an upper bound on $n$ nor a growth condition on $f$.

In many papers, the authors assume box control constraints in the formulation of the problem \Pb; see, for instance,
\cite{CasasHerzogWachsmuth2017,Casas-Mateos2020,Casas-Troltzsch2021,Neitzel-Vexler2012}, \cite[Chapter 5]{Troltzsch2010}.
That is because bounded controls $u$ lead to solutions $y_u$ of \eqref{E1.1} that are functions of $L^\infty(Q)$.
This boundedness of the state is crucial to derive first and second order optimality conditions for local or global minimizers.
Indeed, the $C^1$ or $C^2$ differentiability of the superposition operator $y \to f(y)$  for highly nonlinear functions $f$ requires the boundedness of $y$.
Moreover, as far as we know, the well posedness of the state equation \eqref{E1.1} has not been studied for controls $u \in L^2(Q)$.
In some recent papers, see \cite{Casas2018,Casas-Kunisch2022,CMR2017}, the existence of global minimizers to \Pb in $L^\infty(Q)$ has been proven
in the absence of control constraints or for unbounded control sets with the restriction $n\le 3$ on the dimension. \tcr{The novelties of our paper with respect to these previous results are the following: first we prove that the state equation \eqref{E1.1} is well posed for $L^2(Q)$ controls, and the associated control problem \Pb has at least one global minimizer $\bar u$ in $L^2(Q)$; second we prove that any local minimizer of \Pb in the $L^2(Q)$ sense is an element of $L^\infty(Q)$. Usually, this regularity follows from the optimality conditions satisfied by $\bar u$, but we cannot get such conditions due to the lack of differentiability of the mapping $y \to f(y)$, since the boundedness of the state $\bar y$ corresponding to $\bar u$ cannot be deduced for $L^2(Q)$ controls. Therefore, our approach is necessarily different of the one used in the previous papers.}

In the second part of the paper, we will prove similar results for a Neumann boundary control problem of a semilinear elliptic equation. \tcr{The approach used for a Neumann boundary control can be applied to the case of a distributed control problem; see Remark \ref{R4.1}.}
\tcr{
Classical results on existence of optimal controls subject to box constraints can be found in \cite[Section 4.4]{Troltzsch2010}.
}
The reader is referred to \cite{CKT2022} for the proof of existence of an optimal control in $L^\infty(\Omega)$ for distributed control problems of arbitrary space dimension
\tcr{
without box constraints.
However, in \cite{CKT2022} the analysis of the state equation for the controls in $L^2(\Omega)$ is not performed and, consequently, the existence of minimizers in
$L^2(\Omega)$ is not proven, which are ultimately functions of $L^\infty(\Omega)$.
}

The plan of this paper is as follows.
In section \ref{S2} we investigate \Pb.
First, we analyze the well posedness of the state equation \eqref{E1.1}, see Section \ref{S2.1}, Theorem \ref{T2.1},
where for every control $u \in L^2(Q)$ the existence and uniqueness of a solution in $W(0,T)$ is established.
We also provide an example showing that the  state associated with a control of $L^2(Q)$ does not need to be a bounded function if $n > 1$, see Section \ref{S2.2}.
In Section \ref{S2.3}, we prove that \Pb has at least one global minimizer $\bar u$ in $L^2(Q)$.
Then, we demonstrate that any local or global minimizer of \Pb is an element of $L^\infty(Q)$ in Section \ref{S2.4}.
In the last section of the paper, the same study is applied to a Neumann boundary control problem for a semilinear elliptic equation in dimension $n > 2$.

\section{Optimal distributed control of a semilinear parabolic equation}
\label{S2}

\subsection{Analysis of the State Equation}
\label{S2.1}

We make the following assumptions on the equation \eqref{E1.1},
which are assumed to hold throughout the section.

{\bfseries (A1)}
We assume that $\Omega$ is an bounded domain in $\mathbb{R}^n$, $n \ge 2$, with boundary denoted by $\Gamma$, and $A$ denotes a second-order elliptic operator in $\Omega$ of the form
\[
Ay=-\sum_{i,j=1}^{n}\partial_{x_j}(a_{ij}(x)\partial_{x_i}y) + a_0(x)
\]
with coefficients $a_{ij}, a_0 \in L^\infty(\Omega)$ satisfying for some $\Lambda_A > 0$
\[
\Lambda_A|\xi|^2\leq\sum_{i,j=1}^{{n}} a_{ij}(x)\xi_i\xi_j\ \quad
\forall \xi\in\mathbb{R}^n \text{ and } a_0(x) \ge {0}\ \mbox{ for a.e. } x\in\Omega.
\]

{\bfseries (A2)}
$f: \mathbb{R} \longrightarrow \mathbb{R}$ is function of class $C^1$ satisfying that
\begin{equation}
f(0) = 0\ \text{ and } \ \exists\Lambda_f \ge 0 \text{ such that } f'(s) \ge -\Lambda_f\ \forall s \in \mathbb{R}.
\label{E2.1}
\end{equation}

{\bfseries (A3)}
$\alpha>0$,
$y_0 \in L^\infty(\Omega)$, $y_d \in L^p(0,T;L^q(\Omega))$ with $p, q \in [2,\infty]$ and $\frac{1}{p} + \frac{n}{2q} < 1$.


\tcr{
For convenience, we work with the norm  $\|y\|_{H_0^1(\Omega)} := \|\nabla y \|_{L^2(\Omega)}$.
}
As usual, we denote $W(0,T) =L^2(0,T;H_0^1(\Omega)) \cap H^1(0,T;H^{-1}(\Omega))$.
Then, we have the following existence and uniqueness result for a solution to \eqref{E1.1}.

\begin{theorem}
\label{thm_exist_sol_stateeq_l2}
For every $u \in L^2(Q)$, equation \eqref{E1.1} has a unique solution $y_u \in W(0,T)$.
Moreover, $f(y) \in L^2(Q)$, and there exists a constant $C$ \tcr{depending on $\|y_0\|_{L^\infty(\Omega)}$, but} independent of $u$ such that
\begin{equation}
\|y_u\|_{W(0,T)} + \|f(y_u)\|_{L^2(Q)} \le C(\|u\|_{L^2(Q)} + \|y_0\|_{L^\infty(\Omega)}).
\label{E2.2}
\end{equation}
In addition, if $u_k \rightharpoonup u$ in $L^2(Q)$, then $y_{u_k} \rightharpoonup y_u$ in $W(0,T)$ and $f(y_{u_k}) \rightharpoonup f(y)$ in $L^2(Q)$ hold.
\label{T2.1}
\end{theorem}

\begin{proof}
For every integer $k \ge \|y_0\|_{L^\infty(\Omega)}$ we set $f_k(s) = f\big(\tcr{P_k}(s)\big)$ with $\tcr{P_k}(s) = \min\{\max\{-k,s\},+k\}$.
By a standard application of Schauder's fixed point theorem we infer the existence of a function $y_k \in W(0,T)$ satisfying
\begin{equation}
\left\{\begin{array}{l}\displaystyle\frac{\partial y_k}{\partial t} + Ay_k + f_k(y_k) =  u \ \mbox{ in } Q,\\y_k  =  0\ \mbox{ on } \Sigma,\ \ y_k(0,x) = y_0(x)\ \text{ in } \Omega;\end{array}\right.
\label{E2.3}
\end{equation}
\tcr{
see, for instance, \cite{Casas-Chrysafinos2018} or \cite[Theorem 5.5]{Troltzsch2010}.}
Moreover, testing \eqref{E2.3} with $\e^{-2\Lambda_fs}y_k(s)$ and integrating with respect to $s$, we infer for every $t \in (0,T]$
\begin{multline*}
\frac{1}{2}\e^{-2\Lambda_ft}\|y_k(t)\|^2_{L^2(\Omega)} + \Lambda_f\int_0^t\e^{-2\Lambda_fs}\|y_k(s)\|^2_{L^2(\Omega)}\ds
+ \e^{-2\Lambda_fT} \Lambda_A\|y_k\|^2_{L^2(0,t;H_0^1(\Omega))}\\
 + \int_0^t\int_\Omega\e^{-2\Lambda_fs}f_k(y_k)y_k\dx\ds \\
\begin{aligned}
&\le \int_0^t\int_\Omega\e^{-2\Lambda_fs} uy_k\dx  \color{black}\ds + \frac{1}{2}\|y_0\|^2_{L^2(\Omega)}\\
 &\le C_\Omega\|u\|_{L^2(0,t;L^2(\Omega))}\|y_k\|_{L^2(0,t;H_0^1(\Omega))}
 \color{black}{
 + \frac{1}{2}\|y_0\|^2_{L^2(\Omega)}
 }
 \\
& \le \frac{C_\Omega^2}{2\Lambda_A}{\e^{2\Lambda_fT}}\|u\|^2_{L^2(0,t;L^2(\Omega))} + \frac{\Lambda_A}{2}{\e^{-2\Lambda_fT}}\|y_k\|^2_{L^2(0,t;H^1_0(\Omega))}  + \frac{1}{2}\|y_0\|^2_{L^2(\Omega)}.
\end{aligned}
\end{multline*}
\tcr{
With \eqref{E2.1}
and the mean value theorem,
}
we get that $f_k(y_k)y_k \ge -\Lambda_fy_k^2$.
Inserting this
\tcr{
lower bound into the
}
inequality above, we obtain that the sum of the second and fourth integrals of the left hand side is non negative,
\tcr{
i.e.,
\[ \Lambda_f\int_0^t\e^{-2\Lambda_fs}\|y_k(s)\|^2_{L^2(\Omega)}\ds
 + \int_0^t\int_\Omega\e^{-2\Lambda_fs}f_k(y_k)y_k\dx\ds  \ge0.
\]
}
This leads to
\begin{equation}
\|y_k\|_{L^\infty(0,T;L^2(\Omega))} + \|y_k\|_{L^2(0,T;H_0^1(\Omega))} \le C_1\Big(\|u\|_{L^2(Q)} + \|y_0\|_{L^2(\Omega)}\Big),
\label{E2.4}
\end{equation}
where $C_1$ is independent of $u$ and $y_0$.
Now, we prove that $\{f_k(y_k)\}_{k=1}^\infty$ is bounded in $L^2(Q)$.
To this end, we test \eqref{E2.3} with $f_k(y_k)$ and integrate in $Q$
\begin{align}
\int_0^T\langle\frac{\partial y_k}{\partial t},f_k(y_k)\rangle\dt & - \Lambda_fC\|y_k\|^2_{L^2(0,T;H_0^1(\Omega))} + \|f_k(y_k)\|^2_{L^2(Q)}\notag\\
& \le \frac{1}{2}\|u\|^2_{L^2(Q)} + \frac{1}{2}\|f_k(y_k)\|^2_{L^2(Q)},\label{E2.5}
\end{align}
where $\langle\cdot,\cdot\rangle$ denotes the duality between $H^{-1}(\Omega)$ and $H_0^1(\Omega)$.
We define the function $F_k(\rho) = \int_0^\rho f_k(s)\ds$ for \tcr{ $\rho \in \mathbb{R}$.}
Then, we have
\[
\int_0^T\langle\frac{\partial y_k}{\partial t},f_k(y_k)\rangle\dt  = \int_0^T\frac{\mathrm d}{\dt}\int_\Omega F_k(y_k)\dx\dt = \int_\Omega F_k(y_k(T))\dx - \int_\Omega F_k(y_0)\dx.
\]
By the mean value theorem we get
\tcr{
a function $\theta: \mathbb{R} \longrightarrow [0,1]$ such that for
}
$\rho > 0$
\[
\tcr{F_k(\rho)  = \int_0^\rho f_k(s)\ds = \int_0^\rho f'\big(\theta(s)P_k(s)\big)P_k(s)\ds \ge -\Lambda_f\int_0^\rho s\ds =  -\Lambda_f\frac{\rho^2}{2}.}
\]
\tcr{
We establish the same inequality for $\rho < 0$:
\[
F_k(\rho)  = \int_0^\rho f_k(s)\ds = - \int_\rho^0 f'\big(\theta(s)P_k(s)\big)P_k(s)\ds \ge \Lambda_f\int_\rho^0 s\ds =  -\Lambda_f\frac{\rho^2}{2}.
\]
}
Moreover, since $k \ge \|y_0\|_{L^\infty(Q)}$ we have
\begin{align*}
|F_k(y_0(x))| &\le \left|\int_0^{y_0(x)}f(s)\ds\right| \le \|y_0\|_{L^\infty(Q)}\max\{|f(s)| : |s| \le \|y_0\|_{L^\infty(Q)}\}\\
& = C_{f,y_0}\|y_0\|_{L^\infty(Q)}.
\end{align*}
From the last two estimates we infer
\[
\int_0^T\langle\frac{\partial y_k}{\partial t},f_k(y_k)\rangle\dt \ge -\frac{\Lambda_f}{2}
\color{black}
\|y_k(T)\|^2_{L^2(\Omega)} - C_{f,y_0}\|y_0\|_{L^\infty(\Omega)}.
\]
Using this fact in \eqref{E2.5} we obtain with \eqref{E2.4}
\[
\|f_k(y_k)\|_{\tcr{L^2(Q)}}\le C_2\Big(\|u\|_{L^2(Q)} + \|y_0\|_{L^\infty(\Omega)}\Big).
\]
Hence, from \eqref{E2.3}, \eqref{E2.4}, and this estimate we deduce that $\{y_k\}_{k=1}^\infty$ is bounded in $W(0,T)$.
Therefore, we can take a subsequence, denoted in the same way, such that $y_k \rightharpoonup y$ in $W(0,T)$, $y_k(x,t) \to y(x,t)$ for almost all $(x,t) \in Q$, and $f_k(y_k) \rightharpoonup f(y)$ in $L^2(Q)$.
Then, we can pass to the limit in \eqref{E2.3} and deduce that $y = y_u$ is a solution of \eqref{E1.1}.
Moreover, \eqref{E2.2} follows from the estimates established for $y_k$.
The uniqueness is obtained in the standard way.
Indeed, if $y_1$ and $y_2$ are two solutions of \eqref{E1.1} such that $f(y_i) \in L^2(Q)$ for $i = 1, 2$, then we test \eqref{E1.1} with $\e^{-2\Lambda_f t}(y_2 - y_1)$ and, arguing as above, we deduce that $y_2 - y_1 = 0$.
Finally, the convergence property stated in the theorem follows easily from the estimate \eqref{E2.2}.
\end{proof}

Let us remark that the crucial part of the proof was to establish
the uniform boundedness of $\{f_k(y_k)\}_{k=1}^\infty$ in $L^2(Q)$, which was used to
establish the boundedness of $\{y_k\}_{k=1}^\infty$ in $W(0,T)$.
Here, the assumptions {\bfseries (A2)} on $f$ were essential.

Now, we prove some extra $L^p(Q)$ regularity of the solution $y_u$.
First we state the following lemma.

\begin{lemma} \label{lem_embed_w0t} \label{L2.3}
The following properties are satisfied:

\begin{enumerate}
\renewcommand{\theenumi}{\Roman{enumi}}
 \item\label{lem_23_I}
 The space $L^2(0,T;H^1_0(\Omega))\cap L^\infty(0,T;L^2(\Omega))$ is continuously embedded in $L^p(Q)$ with $p=\frac{2(n+2)}n$.

 \item\label{lem_23_II}
If $u\in L^2(Q)$ and $y_0\in H^1_0(\Omega) \cap L^\infty(\Omega)$, then $y_u \in H^1(Q)$ holds.

 \item\label{lem_23_III}
 Let \tcr{$\frac{1}{\sigma} + \frac{n}{2\gamma} < 1$ with $\sigma, \gamma \in [2, \infty]$}, and $y_0\in L^\infty(\Omega)$ be given. Then
 there exists a constant $C$ independent of $y_0$ such that for all
\tcr{$u\in L^\sigma(0,T;L^\gamma(\Omega))$} it holds
\begin{equation}
\|y_u\|_{L^\infty(Q)} \le C\big(\|u\|_{\tcr{L^\sigma(0,T;L^\gamma(\Omega))}} + \|y_0\|_{L^\infty(\Omega)}\big).
\label{E2.6}
\end{equation}
\end{enumerate}
\end{lemma}
\begin{proof}
\ref{lem_23_I} -
It is enough to apply the Gagliardo-Nirenberg inequality, see, e.g., \cite[\tcr{p. 125}]{Nirenberg1959}, with $p=\frac{2(n+2)}n$, $a = \frac{2}p$, $r = q =2$, and $m = 0$ to get
\[
\|y\|_{L^p(\Omega)}\le C_1 \|\nabla y\|_{L^2(\Omega)}^{\frac{2}p}\|y\|_{L^2(\Omega)}^{1-\frac{2}p}.
\]
Integrating this inequality on $(0,T)$ implies the claim.

\ref{lem_23_II} -
Since $f(y_u) \in L^2(Q)$ by Theorem \ref{T2.1}, the $H^1(Q)$ regularity follows from the classical results for linear parabolic equations; see, for instance, \cite[Section III.2]{Showalter1997}.

\ref{lem_23_III} -
{\color{black} By the change of variables $\phi=\e^{-\Lambda_ft}y_u$, equation \eqref{E1.1} is transformed in
\[
\left\{\begin{array}{l}\displaystyle\frac{\partial\phi}{\partial t} + A\phi + \hat{f}(t,\phi) =  \e^{-\Lambda_ft}u \ \mbox{ in } Q = \Omega \times (0,T),\\\phi  =  0\ \mbox{ on } \Sigma = \Gamma \times (0,T),\ \ \phi(x,0) = y_0(x)\ \text{ in } \Omega,\end{array}\right.
\]
where $\hat{f}:[0,T] \times \mathbb{R} \longrightarrow \mathbb{R}$ is given by $\hat{f}(t,s) = \Lambda_f s + \e^{-\Lambda_ft}f(\e^{\Lambda_ft}s)$. We note that \eqref{E2.1} implies $\frac{\partial\hat{f}}{\partial s}(t,s) = \Lambda_f + f'(\e^{\Lambda_ft}s) \ge 0$ and $\hat{f}(t,0) = 0$.

We set $\beta = \|u\|_{L^\sigma(0,T;L^\gamma(\Omega))} + \|y_0\|_{L^\infty(\Omega)}$. We assume that $\beta > 0$, otherwise $\phi = 0$ and \eqref{E2.6} holds. We also set $\phi_\beta = \frac{1}{\beta}\phi$, $u_\beta = \frac{1}{\beta}u$, and $y_{0\beta} = \frac{1}{\beta}y_0$. Then, $\phi_\beta$ satisfies the equation
\[
\left\{\begin{array}{l}\displaystyle\frac{\partial\phi_\beta}{\partial t} + A\phi_\beta + \frac{1}{\beta}\hat{f}(t,\phi) =  \e^{-\Lambda_ft}u_\beta \ \mbox{ in } Q = \Omega \times (0,T),\\\phi_\beta  =  0\ \mbox{ on } \Sigma = \Gamma \times (0,T),\ \ \phi_\beta(x,0) = y_{0\beta}(x)\ \text{ in } \Omega.\end{array}\right.
\]
Let $k \ge 1$ be given. Define $\phi_{\beta,k} = \phi_\beta - P_k(\phi_\beta)$. Testing the above equation with $\phi_{\beta,k}$, integrating in $(0,t)$ with $t \in (0,T)$, and using that $\frac{\partial\phi_\beta}{\partial t}\phi_{\beta,k} =  \frac{\partial\phi_{\beta,k}}{\partial t}\phi_{\beta,k}$, $\nabla\phi_\beta\cdot\nabla\phi_{\beta,k} = |\nabla\phi_{\beta,k}|^2$, and $\hat{f}(t,\phi(x,t))\phi_{\beta,k}(x,t) \ge 0$, we infer
\[
\frac{1}{2}\|\phi_{\beta,k}(t)\|_{L^2(\Omega)}^2 + \Lambda_A\|\phi_{\beta,k}\|^2_{L^2(0,t;H_0^1(\Omega))} \le \int_0^t\int_\Omega \e^{-\Lambda_fs}u_\beta \phi_{\beta,k}\dx\ds.
\]
The proof now follows the lines of the one of \cite[Theorem III.7.1]{Lad-Sol-Ura68} to deduce the existence of a constant $C > 0$ independent of $(u,y_0)$ such that $\|\phi_\beta\|_{L^\infty(Q)} \le C$. Therefore, we have
\[
\|\phi\|_{L^\infty(Q)} = \beta\|\phi_\beta\|_{L^\infty(Q)} \le C(\|u\|_{L^\sigma(0,T;L^\gamma(\Omega))} + \|y_0\|_{L^\infty(\Omega)}),
\]
which implies \eqref{E2.6}.}
\end{proof}

\begin{theorem}
	\label{thm_stateeq_lp_estimates}
	Let $u\in L^r(Q)$ and $y_0\in L^\infty(\Omega)$ be given such with $r\in [2, 1+ \frac n2]$.
Then the solution $y_u$ of \eqref{E1.1} belongs to $L^q(Q)$, where $q$ has to be chosen as follows:
\begin{enumerate}
	\item if $r< 1+ \frac n2$ then
\begin{equation}\label{eq_cond_pq}
	q =r\frac{n+2}{n+2-2r} \ge r\frac{n+2}{n},
\end{equation}
	\item if $r = 1 + \frac n2$ then  $q< + \infty$ is arbitrary.
\end{enumerate}
	In particular,
	there exists $C = C(q,r)>0$ independent of $u$ and $y_0$ such that
\begin{equation}
	\|y_u\|_{L^q(Q)} \le C( \|u\|_{L^r(Q)} + \|y_0\|_{L^\infty(\Omega)}).
\label{E2.8}
\end{equation}
\end{theorem}

\begin{proof}
For $r< 1+ \frac n2$ we set $p =\frac{rn}{n+2-2r}$.
Due to the assumptions on $r$, it follows {that} $n\ge 3$ and $p\ge r\frac n{n-2} \ge r \ge 2$.
In the critical case $r = 1+ \frac n2$, we can choose $p\ge 2$ arbitrarily.
Then in both cases, we have $p\ge 2$ and $p$ satisfies
\begin{equation}\label{eq_cond_pr}
	\frac 1r + \frac{p-1}p \frac{n}{n+2} \le 1.
\end{equation}
In addition, \eqref{eq_cond_pq} yields $q =p\frac{n+2}n$.

Throughout the proof we abbreviate $y:=y_u$.

{\em 1.
Estimates for regular $y$.}
Let us assume for the moment that $y_t \in L^2(Q)$ and $y\in L^\infty(Q)$.
\tcr{
Then, we have that $|y|^{p-2}y \in H^1(Q) \cap L^\infty(Q)$.
}
Note that
\begin{multline*}
\int_\Omega \sum_{i,j = 1}^na_{ij}\partial_{x_i}y\partial_{x_j}(|y|^{p-2}y) \dx
  = \int_\Omega (p-1)\sum_{i,j = 1}^na_{ij}\partial_{x_i}y\partial_{x_j}y \cdot |y|^{p-2}  \dx \\
  \begin{aligned}
&=\int_\Omega \frac{4(p-1)}{p^2} \sum_{i,j = 1}^na_{ij}\partial_{x_i}(|y|^{p/2})\partial_{x_j}(|y|^{p/2}) \dx\\
& \ge \frac{4(p-1)}{p^2}\Lambda_A \int_\Omega |\nabla (|y|^{\frac p2})|^2 \dx.
\end{aligned}
\end{multline*}
Taking $|y|^{p-2}y$ as test function in the weak formulation of \eqref{E1.1},
\tcr{
integrating on $(0,t)\times \Omega$, and using the above inequality,
}
results in
	\begin{multline*}
		\frac1p ( \|y(t)\|_{L^p(\Omega)}^p - \|y_0\|_{L^p(\Omega)}^p)
		+  \frac{4(p-1)}{p^2}\Lambda_A\int_0^t\int_\Omega |\nabla (|y|^{\frac p2})|^2 \dx\ds\\
		+ \int_0^t \int_\Omega f(y)y|y|^{p-2}\dx\ds  \le \int_0^t \int_\Omega uy|y|^{p-2}\dx\ds.
	\end{multline*}
Since $f(y)y \ge -\Lambda_f y^2$, we obtain
\begin{multline*}
    \frac1p  \|y(t)\|_{L^p(\Omega)}^p
		+ \frac{4(p-1)}{p^2}\Lambda_A\int_0^t\int_\Omega |\nabla (|y|^{\frac p2})|^2 \dx\ds
		\\
		\le \int_Q |u|\cdot |y|^{p-1}\dx\ds + \frac1p \|y_0\|_{L^p(\Omega)}^p
		+\Lambda_f\int_0^t\|y(t)\|_{L^p(\Omega)}^p \ds.
\end{multline*}
By Gronwall inequality, we obtain
\[
	\|y\|_{L^\infty(0,T;L^p(\Omega))}^p  + \|\nabla (|y|^{\frac p2})\|_{L^2(Q)}^2 \le C_1\left( \int_Q |u|\cdot |y|^{p-1}\dx\ds + \|y_0\|_{L^p(\Omega)}^p \right),
\]
which is an estimate of $|y|^{\frac p2}$ in $L^\infty(0,T;L^2(\Omega)) \cap L^2(0,T;H^1(\Omega))$.
Using Lemma \ref{lem_embed_w0t}-\ref{lem_23_I}, this space embeds continuously  into $L^{\frac{2(n+2)}n}(Q)$, which implies $y\in L^{p\frac{n+2}n}(Q)=L^q(Q)$ together with the corresponding estimate
\[
	\|y\|_{L^{q}(Q)}^p \le C_2\left( \int_Q |u|\cdot |y|^{p-1}\dx\ds + \|y_0\|_{L^p(\Omega)}^p \right).
\]
Due to the property \eqref{eq_cond_pr},
we can apply H\"older and Young inequalities, and we get
\[
	\|y\|_{L^{q}(Q)}^p \le C_3( \|u\|_{L^r(Q)}^p + \|y_0\|_{L^p(\Omega)}^p),
\]
which is the claim.
In the critical case $r=1+\frac n2$, we can chose $p$ and thus $q$ arbitrarily large.
In any case, the constant $C$ in the inequality \eqref{E2.8} depends on $p$ and $q$.

{\em 2.
General case.} Given $u \in L^r(Q)$ we set $u_k = P_k(u)$.
For $y_0 \in L^\infty(\Omega)$, we take a sequence $\{\hat y{_{0k}}\}_{k = 1}^\infty \subset H_0^1(\Omega)$ such that $\hat y_{0k}(x) \to y_0(x)$ for almost every $x \in \Omega$.
Now, we define $y_{0k} = \tcr{P_{M_0}}(\hat y_k)$ with $M_0 = \|y_0\|_{L^\infty(\Omega)}$.
We still have that $\{y_{0k}\}_{k = 1}^\infty \subset H_0^1(\Omega)$ and $\|y_{0k}\|_{L^\infty(\Omega)} \le \|y_0\|_{L^\infty(\Omega)}$.
Then, the solution $y_k$ of \eqref{E1.1} associated with $(u_k,y_{0k})$ is an element of $H^1(Q) \cap L^\infty(Q)$; see Lemma \ref{lem_embed_w0t}.
From Theorem \ref{T2.1} we infer that $y_k \rightharpoonup y$ in $W(0,T)$.
Moreover, every function $y_k$ satisfies the inequality \eqref{E2.8} with $y_k$ in the left hand side and $u_k$ and $y_{0k}$ on the right.
Now, it is easy to pass to the limit in this inequality and to deduce that $y$ satisfies \eqref{E2.8} as well.
\end{proof}

The reader is referred to \cite{Troltzsch2000} for other $L^p$ estimates in the case of linear equations, which were proven using semigroup theory.

\begin{remark}
	The regularity $y_0 \in L^\infty(\Omega)$ was used in the proof to be able to perform the approximation procedure in the second part,
	as the existence of $L^\infty(Q)$ solutions for the nonlinear equation requires this regularity of $y_0$.
	The estimates themselves only used $L^p$-norms of $y_0$, $p<\infty$.
	\label{R2.2}
\end{remark}

\subsection{Example}
\label{S2.2}
Let us show by means of a small counterexample that the solution of \eqref{E1.1} is not necessarily an element of $L^\infty(Q)$ if the control $u$ is just an element of $L^2(Q)$.
Actually, we prove something more general: for $n \ge 2$ and smooth domain $\Omega$ the space $L^2(0,T;H^2(\Omega) \cap H_0^1(\Omega)) \cap H^1(0,T;L^2(\Omega))$ is not contained in $L^\infty(Q)$.

For $r,s>0$, let $Q_{r,s}:= B_r(0) \times [1-s,1+s]\subset \R^{n+1}$, where $B_r(0)$ is the open ball of radius $r$.
Let us choose $\phi\in C_c^\infty(\R^{n+1})$ such that $0 \le\phi(x,t) \le 1$, $\phi=1$ on $Q_{1,1}$, and $\phi=0$ on $\R^{n+1}\setminus Q_{2,2}$.
We set $Q = \Omega \times (0,T) = B_2(0) \times (0,2)$ and define the function $y$ in $Q$ by
\[
	y(x,t) := \sum_{k=1}^\infty k^{-1} \phi ( 2^kx, 2^{2k}(t-1)).
\]
Note that for $(x,t)\ne (0,1)$ only finitely many summands are non-zero.
The derivatives of $(x,t) \mapsto \phi ( 2^kx, 2^{2k}(t-1))$ are supported on
$Q_{2^{1-k},2^{1-2k}} \setminus Q_{2^{-k},2^{-2k}}$, hence the supports of the derivatives of the terms in the sum are disjoint.
Due to this fact, and using the coordinate transform $(\hat x,\hat t) = (2^kx, 2^{2k}(t - 1))$, we deduce
\[
	\|\partial_t y\|_{L^2(Q)}^2 = \|\partial_t y\|_{L^2(\mathbb R^{n + 1})}^2
	= \sum_{k=1}^\infty k^{-2} 2^{k(2-n)}	\|\partial_t \phi\|_{L^2(\mathbb R^{n + 1})}^2 < +\infty
\]
and similarly
\[
	\|\partial_{x_i}\partial_{x_j}y\|_{L^2(Q)}^2 =\|\partial_{x_i}\partial_{x_j}y\|_{L^2(\mathbb R^{n + 1})}^2 = \sum_{k=1}^\infty k^{-2} 2^{k(2-n)}\|\partial_{x_i}\partial_{x_j}\phi\|_{L^2(\mathbb R^{n + 1})}^2 < +\infty.
\]
Since $y$ vanishes in $Q \setminus Q_{1,1}$, it follows $y\in L^2(0,T;H^2(\Omega) \cap H_0^1(\Omega)) \cap H^1(0,T;L^2(\Omega))$ and $y(x,0) = 0$.
For $m\in \mathbb N$, let $(x,t)\in Q_{2^{-m},2^{-2m}}$.
Then
\[
	y(x,t) \ge \sum_{k=1}^m k^{-1} \phi ( 2^kx, 2^{2k}(t-1)) =\sum_{k=1}^m k^{-1}.
\]
Clearly, $ Q_{2^{-m},2^{-2m}}$ has positive measure, and $y\not\in L^\infty(Q)$.
Now, setting $u = \frac{\partial y}{\partial t} - \Delta y$, we infer that $y$ is the unique solution of
\[
\left\{\begin{array}{l}\displaystyle\frac{\partial y}{\partial t} - \Delta y =  u \ \mbox{ in } Q,\\y  =  0\ \mbox{ on } \Sigma,\ \ y(x,0) = 0\ \text{ in } \Omega.\end{array}\right.
\]
Moreover, since $L^2(0,T;H^2(\Omega) \cap H_0^1(\Omega)) \cap H^1(0,T;L^2(\Omega)) \subset C([0,T];H_0^1(\Omega))$ and $H_0^1(\Omega) \subset L^6(\Omega)$ if $n \le 3$, for $f(y) = y^3$ we have that
\[
\|f(y)\|_{L^2(Q)} \le C\|y\|^2_{C([0,T];H_0^1(\Omega))}\|y\|_{L^2(0,T;H_0^1(\Omega))} < \infty.
\]
Therefore, $u = \frac{\partial y}{\partial t} - \Delta y + f(y) \in L^2(Q)$ for $n = 2$ or 3, and $y \not\in L^\infty(Q)$ solves the equation \eqref{E1.1}.

\subsection{Existence of solutions in $L^2(Q)$}
\label{S2.3}
In this section, we prove the existence of at least one solution to \Pb.
Below we will prove that any local solution of \Pb belongs to $L^\infty(Q)$.
Here, local solutions are intended in the sense of $L^2(Q)$.
Let us start proving the existence of optimal controls in $L^2(Q)$.
The proof is standard, and we only give a brief sketch.

\begin{theorem}\label{thm_exist_sol_P}
	Problem \Pb admits a global solution.
\end{theorem}
\begin{proof}
	Due to the structure of the cost functional $J$, a minimizing sequence $\{u_k\}_{k = 1}^\infty$ is bounded in $L^2(Q)$,
	and hence we can assume (after passing to a subsequence if necessary) that $u_k \rightharpoonup \bar u$ in $L^2(Q)$.
	Due to Theorem \ref{T2.1}, we can pass to the limit in the state equation.
	Using the weak sequentially lower semicontinuity of the cost functional $J$, we can prove that $\bar u$ is a global solution of \Pb.
\end{proof}

\subsection{Local solutions are in $L^\infty(Q)$}
\label{S2.4}

In order to prove that local solutions of \Pb are in $L^\infty(Q)$, we employ
the following auxiliary problems, which are localized and contains box constraints parametrized by $M$.
Let a local minimizer $\bar u$ of \Pb be given.
Let $\rho>0$ {be} such that $J(\bar u) \le J(u)$ for all $u$ with $\|u-\bar u\|_{L^2(Q)} \le \rho$.
We define the following problem:
\begin{equation}\label{eq_PbM}
	\tag{P${}_M$}
	\min J(u) + \frac12 \|u-\bar u\|_{L^2(Q)}^2
\end{equation}
subject to $\|u-\bar u\|_{L^2(Q)} \le\rho$, $|u(x,t)| \le M$  f.a.a.\@ $(x,t) \in Q$.

Similar to Theorem \ref{thm_exist_sol_P}, we obtain solvability of \PbM.

\begin{lemma}
	Let $\{u_M\}_{M>0}$ be a family of solutions of \PbM.
Then $u_M \to \bar u$ in $L^2(Q)$ for $M\to \infty$.
	\label{L3.1}
\end{lemma}
\begin{proof}
Let $M_k \to \infty$ and set $u_k:=u_{M_k}$.
We can assume (after passing to a subsequence if necessary) that $u_k \rightharpoonup u^*$ in $L^2(Q)$.
Let us define the truncation $\bar u_k = \tcr{P_{M_k}}(\bar u)$.
Then $\bar u_k \to \bar u$ in $L^2(Q)$.
Hence, $\bar u_k$ is a feasible control for problem \PbMk for $k$ large enough and, consequently $J(u_k) + \frac12 \|u_k-\bar u\|_{L^2(Q)}^2 \le J(\bar u_k)+ \frac12 \|\bar u_k-\bar u\|_{L^2(Q)}^2$.
Due to the weak lower semicontinuity of $J$ on $L^2(Q)$, we can pass to the limit in this inequality to obtain
$J(u^*) + \frac12 \|u^*-\bar u\|_{L^2(Q)}^2 \le J(\bar u)$.
Since $\|u^*-\bar u\|_{L^2(Q)} \le\rho$, it follows $\bar u=u^*$ by
the optimality of $\bar u$ in the ball $B_\rho(\bar u)$.
By the properties of limit inferior and superior, we have
\[\begin{split}
	J(\bar u) &= \lim_{k\to\infty} \left( J(\bar u_k)+ \frac12 \|\bar u_k-\bar u\|_{L^2(Q)}^2 \right)\\
	 &\ge \limsup_{k\to\infty} \left(J(u_k) + \frac12 \|u_k-\bar u\|_{L^2(Q)}^2\right)\\
	 &\ge \liminf_{k\to\infty} J(u_k) + \limsup_{k\to\infty}\frac12 \|u_k-\bar u\|_{L^2(Q)}^2\\
	 & \tcr{\ge J(\bar u) + \limsup_{k\to\infty}\frac12 \|u_k-\bar u\|_{L^2(Q)}^2}\\
	 & \ge J(\bar u) + \liminf_{k\to\infty}\frac12 \|u_k-\bar u\|_{L^2(Q)}^2 \ge J(\bar u).
\end{split}\]
Hence $\|u_k-\bar u\|_{L^2(Q)}^2\to 0$.
Since the limit is independent of the chosen subsequence, the claim follows.
\end{proof}

From this lemma we infer the existence of $M_0$ such that $\|u_M - \bar u\|_{L^2(Q)} < \rho$ for all $M > M_0$.
Hence, $u_M$ is a local minimizer of $J(u) + \frac12 \|u-\bar u\|_{L^2(Q)}^2$ on the set of controls of $u \in L^2(Q)$ such that $|u| \le M$.
Since
the set of feasible controls for \PbM is bounded in $L^\infty(Q)$, then
a classical proof \cite[Chapter 5]{Troltzsch2010} establishes the following optimality conditions for the local minimizers.

\begin{theorem}
Let $u_M$ be a local minimizer of \PbM for $M > M_0$.
Then, there exists $\varphi_M\in H^1(Q) \cap L^\infty(Q)$ satisfying
\begin{align}
&\left\{\begin{array}{l}\displaystyle -\frac{\partial \varphi_M}{\partial t} + A^*\varphi_M + f'(y_M)\varphi_M =  y_M - y_d \ \mbox{ in } Q,\\\varphi_M  =  0\ \mbox{ on } \Sigma,\ \ \varphi_M(x,T) = 0\ \text{ in } \Omega,\end{array}\right.
\label{E3.1}\\
&\label{E3.2}
	\int_Q(\varphi_M + \alpha u_M + u_M-\bar u)(v-u_M)\dx\dt \ge 0 \quad \forall v\in L^2(Q): \ |v| \le M,
\end{align}
where $y_M$ is the state associated with $u_M$ and
\[
A^*\varphi =-\sum_{i,j=1}^{n}\partial_{x_j}(a_{ji}(x)\partial_{x_i}\varphi) + a_0(x)\varphi.
\]
\label{T3.1}
\end{theorem}

From \eqref{E3.1} and due to
\tcr{
$y_d \in L^p(0,T;L^q(\Omega))$ with $p, q \in [2,\infty]$ and $\frac{1}{p} + \frac{n}{2q} < 1$
},
the boundedness of $\varphi_M$ follows from \tcr{\cite[Theorem III.7.1]{Lad-Sol-Ura68}}.
The $H^1(Q)$ regularity is classical; see \cite[Section III.2]{Showalter1997}.

\begin{theorem}\label{thm_adj_bounded}
Let $\bar u$ be a local minimizer of \Pb.
Then $\bar u\in L^\infty(Q)$ holds.
\end{theorem}
\begin{proof}
From Lemma \ref{L3.1} we know that there exists a number $M_0 > 0$ and a family $\{u_M\}_{M > M_0}$ of local minimizers of problems \PbM such that \eqref{E3.1}--\eqref{E3.2} hold and $u_M \to \bar u$ in $L^2(Q)$ as $M \to \infty$.
Denote by $y_M$ the state associated with $u_M$.
From \eqref{E3.1} we deduce that $\{\varphi_M\}_{M > M_0}$ is bounded in $W(0,T)$.
Hence, there exists a sequence $\{M_k\}_{k = 1}^\infty$ converging to infinity and a function $\varphi \in W(0,T)$ such that $\varphi_k = \varphi_{M_k} \rightharpoonup \varphi$ in $W(0,T)$.
\tcr{
Due to the compactness of the embedding $W(0,T) \subset L^2(Q)$ \cite[Theorem 5.1]{Lions69},
}
we have that $\varphi_k \to \varphi$ in $L^2(Q)$.
Let us denote $u_k =u_{M_k}$ and $y_k = y_{M_k}$.
Taking a new subsequence, we can also assume that $\varphi_k(x,t) \to \varphi(x,t)$ and $u_k(x,t) \to \bar u(x,t)$ for almost all $(x,t) \in Q$.

Now, from \eqref{E3.2} we infer
\begin{equation}
u_k = \tcr{P_{M_k}}\big(-\frac{1}{\alpha}[\varphi_k + u_k - \bar u]\big).
\label{E3.3}
\end{equation}
Passing pointwise to the limit in the above identity we deduce that $\bar u = -\frac{1}{\alpha}\varphi$.
We are going to prove that $\varphi \in L^\infty(Q)$.
First, the equation \eqref{E3.1} is split in two equations
\begin{equation}
\left\{\begin{array}{l}\displaystyle -\frac{\partial \phi_k}{\partial t} + A^*\phi_k + f'(y_k)\phi_k =  y_k \ \mbox{ in } Q,\\\phi_k  =  0\ \mbox{ on } \Sigma,\ \ \phi_k(x,T) = 0\ \text{ in } \Omega\end{array}\right.
\label{E3.4}
\end{equation}
and
\begin{equation}
\left\{\begin{array}{l}\displaystyle -\frac{\partial \psi_k}{\partial t} + A^*\psi_k + f'(y_k)\psi_k =  y_d \ \mbox{ in } Q,\\\psi_k  =  0\ \mbox{ on } \Sigma,\ \ \psi_k(x,T) = 0\ \text{ in } \Omega.\end{array}\right.
\label{E3.5}
\end{equation}
Then, we have $\varphi_k = \phi_k - \psi_k$, $\phi_k \rightharpoonup \phi$ and $\psi_k \rightharpoonup \psi$ in $W(0,T)$, and $\varphi = \phi - \psi$.
Due to our assumptions on $y_d$,  we know that $\{\psi_k\}_{k=1}^\infty$ is uniformly bounded in $L^\infty(Q)$;
see Lemma \ref{L2.3}-\ref{lem_23_III}. As a consequence, we get that $\psi \in L^\infty(Q)$.
We are going to prove that $\{\phi_k\}_{k = 1}^\infty$ is also bounded in $L^\infty(Q)$. Since $\{u_k\}_{k=1}^\infty$ is bounded in $L^2(Q)$, we infer from Theorem \ref{thm_stateeq_lp_estimates}
\[
\|y_k\|_{L^{2\frac{n+2}{n-2}}(Q)} \le C\big(\|u_k\|_{L^2(Q)} + \|y_0\|_{L^\infty(Q)}\big) \le C_1 \ \ \forall k \ge 1.
\]
If $n \le 5$ then the inequality $2\frac{n+2}{n-2} > 1 + \frac{n}{2}$ holds.
Therefore, applying again Lemma \ref{L2.3}-\ref{lem_23_III} to the equation \eqref{E3.4}, we deduce the existence of a constant $C_2 > 0$ such that $\|\phi_k\|_{L^\infty(Q)} \le C_2$ for every $k \ge 1$.
This yields $\varphi \in L^\infty(Q)$
\tcr{and $\bar u \in L^\infty(Q)$
}
as well.

For $n > 5$ we can repeat the arguments of Theorem \ref{thm_stateeq_lp_estimates} to the equation \eqref{E3.4} and deduce
\[
\|\phi_k\|_{L^{2\frac{(n+2)^2}{(n-2)^2}}(Q)} \le C\|y_k\|_{L^{2\frac{n+2}{n-2}}(Q)} \le CC_1 \quad \forall k \ge 1.
\]
This implies that $\phi \in L^{2\frac{(n+2)^2}{(n-2)^2}}(Q)$ and, consequently $\bar u = -\frac{1}{\alpha}\varphi \in L^{2\frac{(n+2)^2}{(n-2)^2}}$ holds.
Using \eqref{E3.2} we get
\[
u_k = \tcr{P_{M_k}}\big(\frac{-1}{1 + \alpha}[\varphi_k - \bar u]\big).
\]
This implies
\[
\|u_k\|_{L^{2\frac{(n+2)^2}{(n-2)^2}}(Q)} \le \frac{1}{1 + \alpha}\big(\|\varphi_k\|_{L^{2\frac{(n+2)^2}{(n-2)^2}}(Q)} + \|\bar u\|_{L^{2\frac{(n+2)^2}{(n-2)^2}}(Q)}\big) \le C_3.
\]
A second application of Theorem \ref{thm_stateeq_lp_estimates} yields
\[
\|y_k\|_{L^{2\frac{(n+2)^3}{(n-2)^3}}(Q)} \le C\big(\|u_k\|_{L^{2\frac{(n+2)^2}{(n-2)^2}}} + \|y_0\|_{L^\infty(Q)}\big) \le C_4 = C\big(C_3 + \|y_0\|_{L^\infty(Q)}\big) \ \ \forall k \ge 1.
\]
If $2\frac{(n+2)^3}{(n-2)^3} > 1 + \frac{n}{2}$, then we argue as before to deduce that $\bar u = - \frac{1}{\alpha}\varphi \in L^\infty(Q)$.
If not then we can repeat the arguments and increase the $L^p(Q)$ regularity of $y_k$ until we obtain the desired regularity for $\varphi$ after finitely many steps.
\end{proof}

We proved that any local solution of \Pb is a function belonging to $L^\infty(Q)$.
Hence, the problem \Pb is equivalent to the minimization of $J$ on $ L^\infty(Q)$.
It is well known that the mapping $u \mapsto y_u$ from
$L^\infty(Q)$ to $W(0,T) \cap L^\infty(Q)$ is of class $C^1$.
Then, we can write the necessary optimality conditions satisfied by any local minimizer $\bar u$ of \Pb as follows, see \cite[Chapter 5]{Troltzsch2010}
\begin{align}
&\left\{\begin{array}{l}\displaystyle\frac{\partial\bar y}{\partial t} + A\bar y + f(\bar y) = \bar u \ \mbox{ in } Q,\\\bar y  =  0\ \mbox{ on } \Sigma,\ \ \bar y(x,0) = y_0(x)\ \text{ in } \Omega,\end{array}\right.\label{E3.6}\\
&\left\{\begin{array}{l}\displaystyle -\frac{\partial\bar\varphi}{\partial t} + A^*\bar\varphi + f'(\bar y)\bar\varphi =  \bar y - y_d \ \mbox{ in } Q,\\\bar\varphi  =  0\ \mbox{ on } \Sigma,\ \ \bar\varphi(x,T) = 0\ \text{ in } \Omega,\end{array}\right.\label{E3.7}\\
&\quad\ \bar\varphi + \alpha\bar u = 0, \label{E3.8}
\end{align}
where $\bar y \in W(0,T) \cap L^\infty(Q)$ and $\bar\varphi \in H^1(Q) \cap C^{\mu,\frac{\mu}{2}}(\bar Q)$ for some $\mu \in (0,1)$.
The reader is referred to \cite[Theorem III.10.1]{Lad-Sol-Ura68} for the H\"older regularity of $\bar\varphi$.
Then, as a consequence of \eqref{E3.8}, we deduce that any local solution of \Pb also belongs to $H^1(Q) \cap C^{\mu,\frac{\mu}{2}}(\bar Q)$.

{\color{black}
\begin{remark}
\label{R2.4}
Given a measurable subset $\omega \subset \Omega$ with positive Lebesgue measure, all the results of this paper are valid if we replace $u$ in equation \eqref{E1.1} by $u\chi_\omega$ with $u \in L^2(\omega \times (0,T))$ and $\chi_\omega$ being the characteristic function of $\omega$. The changes in the proofs are obvious.

We also observe that in real world applications the case  $u(x,t) = \sum_{j = 1}^mu_j(t)g_j(x)$ with $\{u_j\}_{j = 1}^m \subset L^2(0,T)$ and supp$(g_i) \cap$supp$(g_j) = \emptyset$ for $i \neq j$ is very interesting.
In this case, if $\{g_j\}_{j = 1}^m \subset L^q(\Omega)$ for $q > n$, we deduce from \eqref{E2.6} that the solution of \eqref{E1.1} belongs to $L^\infty(Q)$. Consequently, the mapping $(u_1,\dots,u_m) \to y$ is differentiable from $L^2(0,T)^m$ to $L^\infty(Q) \cap W(0,T)$. Hence, it is obvious to prove the existence of an optimal control and to deduce the optimality system for every local solution $\{\bar u_j\}_{j = 1}^m$. Moreover, since the states belong to $L^\infty(Q)$, the adjoint states belong to $L^\infty(Q)$ as well. In this context, the optimality condition \eqref{E3.8} is replaced by
\[
\int_\Omega g_j(x)\bar\varphi(x,t)\dx + \alpha \bar u_j(t) = 0 \quad \text{for } 1 \le j \le m \text{ and almost all } t\in (0,T).
\]
This implies that $\{\bar u_j\}_{j = 1}^m \subset L^\infty(0,T)$.
\end{remark}}

\section{Optimal Neumann boundary control of a semilinear elliptic equation}
\label{S3}

In this section we study the following control problem
\begin{equation} \label{eq_Pn} \tag{P${}_\mathrm{ell}$}
 \inf_{u \in L^2(\Gamma)} J(u):= \frac{1}{2}\int_\Omega(y_u - y_d)^2\dx + \frac{\alpha}{2}\int_\Gamma u^2\dx,
\end{equation}
where $y_u$ is the solution of the  semilinear elliptic equation
\begin{equation}
\left\{\begin{array}{l}\displaystyle Ay + f(\cdot, y) =  g \ \mbox{ in } \Omega,\\\partial_{\nu_A}y  =  u\ \mbox{ on } \Gamma.\end{array}\right.
\label{E4.1}
\end{equation}

Here, $\Omega \subset \mathbb{R}^n$ with $n > 2$ is a bounded domain with Lipschitz boundary $\Gamma$.
$A$ denotes the same operator as in section \ref{S2} and $\partial_{\nu_A}y = \sum_{i,j = 1}^na_{ij}(x)\partial_{x_i}y\nu_j(x)$, where $\nu(x)$ is the unit outward normal vector to $\Gamma$ at the point $x$.
We make the following assumptions on \Pbn:

{\bfseries (B1)} The coefficients of the operator $A$ satisfy the conditions in {\bfseries (A1)} with the additional requirement
that $a_0 \not\equiv 0$.

{\bfseries (B2)}
$f: \Omega \times \mathbb{R} \longrightarrow \mathbb{R}$ is a Carath\'eodory function that is of class $C^1$ with respect to the second parameter
satisfying
\begin{equation}
f(x,0) = 0\ \text{ and } \frac{\partial f}{\partial y} (x,y) \ge 0\ \text{ for a.a. } x\in \Omega, \forall y \in \mathbb{R}.
\label{E4.4}
\end{equation}
In addition, for every $M>0$ there is $C_{f,M}>0$ such that $|f(x,y)|  + |\frac{\partial f}{\partial y}(x,y)|\le C_{f,M}$ for almost all $x\in \Omega$ and all $|y|\le M$.

{\bfseries (B3)}
$\alpha>0$, $g,y_d \in L^p(\Omega)$  with $p > \frac{n}{2}$.  

The condition $f(\cdot,0) = 0$ was imposed to shorten the presentation. It can be replaced by the condition $f(\cdot,0) \in L^p(\Omega)$  with $p > \frac{n}{2}$.
In the analysis, we can then replace $f$ and $g$ by $f(\cdot,y)-f(\cdot,0)$ and $g-f(\cdot,0)$.

Analogously to the control problem analyzed in section \ref{S2}, here we will prove that \Pbn is well posed and has at least one global minimizer in $L^2(\Gamma)$.
Then, we establish that any local minimizer of \Pbn in $L^2(\Gamma)$ is actually a function of $L^\infty(\Gamma)$.
This regularity implies the $C(\bar\Omega)$ regularity of the locally optimal states, which allows to derive first and second order optimality conditions for \Pbn.
We recall that, under the above conditions, for $n = 2$ and for every $u \in L^2(\Gamma)$ there exists a unique solution $y_u \in H^1(\Omega) \cap C(\bar\Omega)$.
Therefore, we can differentiate the relation $u \to f(y_u)$ and derive first order optimality conditions for \Pbn.
From these conditions we infer as usual the $C(\bar\Omega)$ regularity of the adjoint state and, consequently, the $C(\Gamma)$ regularity of the locally optimal controls. This is why we have selected $n > 2$ in this section.

\subsection{Analysis of the state equation}
\label{S3.1}
Associated with $A$, we define the bilinear form $B:H^1(\Omega) \times H^1(\Omega) \longrightarrow \mathbb{R}$ by
\[
B(y,z) = \int_\Omega\big(\sum_{i,j = 1}^na_{ij}\partial_{x_i}y\partial_{x_j}z + a_0yz\big)\dx.
\]
From Assumption {\bfseries (B1)}  we get
\begin{equation}
\exists \Lambda_B > 0 \ \text{ such that }\ \Lambda_B\|y\|^2_{H^1(\Omega)} \le B(y,y)\ \ \forall y \in H^1(\Omega).
\label{E4.5}
\end{equation}

In the following, $\langle\cdot,\cdot\rangle_\Omega$ and $\langle\cdot,\cdot\rangle_\Gamma$ denote the duality pairing between $H^1(\Omega)^*$ and $H^1(\Omega)$ and $H^{-\frac{1}{2}}(\Gamma)$ and $H^{\frac{1}{2}}(\Gamma)$, respectively.
\tcr{
Let us first state the existence result for weak solutions of the state equation.
We will give its proof below.
}
\begin{theorem}
Given $u \in H^{-\frac{1}{2}}(\Gamma)$ and $g \in H^1(\Omega)^*$, there exists a unique function $y_u \in H^1(\Omega)$ such that $f(\cdot,y_u) \in L^1(\Omega) \cap H^1(\Omega)^*$ and
\begin{equation}
B(y_u,z) + \langle f(\cdot,y_u),z\rangle_\Omega  = \langle g,z\rangle_\Omega + \langle u,z\rangle_\Gamma\ \  \forall z \in H^1(\Omega).
\label{E4.6}
\end{equation}
Furthermore, if $u_k \to u$ in $H^{-\frac{1}{2}}(\Gamma)$, then $y_{u_k} \to y_u$ in $H^1(\Omega)$ and $f(\cdot,y_{u_k}) \to f(\cdot,y_u)$ in $L^1(\Omega) \cap H^1(\Omega)^*$ hold.
\label{T4.1}
\end{theorem}

According to this result,
we call $y_u \in H^1(\Omega)$ a weak solution of \eqref{E4.1} if $f(\cdot,y_u) \in L^1(\Omega) \cap H^1(\Omega)^*$ and \eqref{E4.6} is satisfied.

\tcr{
If $h \in H^1(\Omega)^*$ and there exists $\phi \in L^1(\Omega)$ such that
\[
\langle h,z\rangle = \int_\Omega \phi(x)z(x)\dx \quad \forall z \in H^1(\Omega) \cap L^\infty(\Omega),
\]
we say that $h \in L^1(\Omega) \cap H^1(\Omega)^*$. In this case, we identify $h$ with $\phi$.
}

\tcr{
If $z \in H^1(\Omega)$ satisfies
}
$hz \in L^1(\Omega)$ we also have that $\langle h,z\rangle_\Omega  = \int_\Omega h(x)z(x)\dx$.
Indeed, define $z_k = P_k(z)$ for every integer $k \ge 1$.
Then, $z_k \in H^1(\Omega) \cap L^\infty(\Omega)$ and $z_k \to z$ in $H^1(\Omega)$ holds.
Moreover, since $hz \in L^1(\Omega)$, $h(x)z_k(x) \to h(x)z(x)$ for almost all $x \in \Omega$, and $|h(x)z_k(x)| \le |h(x)z(x)|$, Lebesgue's dominated convergence theorem implies that $hz_k \to hz$ in $L^1(\Omega)$.
These arguments yield
\[
\int_\Omega h(x)z(x)\dx = \lim_{k \to \infty}\int_\Omega h(x)z_k(x)\dx = \lim_{k \to \infty}\langle h,z_k\rangle_\Omega  = \langle h,z\rangle_\Omega .
\]

\begin{lemma}
The following properties are satisfied:
\begin{enumerate}
\item\label{L4.1-i}  If $f(\cdot,y) \in H^1(\Omega)^* \cap L^1(\Omega)$, then $f(\cdot,y)y \in L^1(\Omega)$ holds.
\item\label{L4.1-ii} If $y, z \in H^1(\Omega)$ and $f(\cdot,y), f(\cdot,z) \in H^1(\Omega)^* \cap L^1(\Omega)$, then the inequality $\langle f(\cdot,y) - f(\cdot,z),y - z\rangle_\Omega  \ge 0$ is fulfilled.
\end{enumerate}
\label{L4.1}
\end{lemma}
\begin{proof}
To prove the first statement, we define $y_k = P_k(y)$ for every integer $k \ge 1$.
Then, we have that $y_k \to y$ in $H^1(\Omega)$, $y_k(x) \to y(x)$ for almost all $x \in \Omega$, and $\{y_k\}_{k = 1}^\infty \subset L^\infty(\Omega)$.
Hence, we also have $f(\cdot,y_k(x)) \to f(\cdot,y(x))$ for almost all $x \in \Omega$.
Moreover, \eqref{E4.4} implies that $f(\cdot,s)s \ge 0$ for every $s \in \mathbb{R}$. Therefore, using Fatou's lemma we get
\begin{align*}
\int_\Omega f(\cdot,y)y\dx &\le \liminf_{k \to \infty} \int_\Omega f(\cdot,y_k)y_k\dx \le \liminf_{k \to \infty}\int_\Omega f(\cdot,y)y_k \dx\\
& = \lim_{k \to \infty}\langle f(\cdot,y),y_k\rangle_\Omega  = \langle f(\cdot,y),y\rangle_\Omega  < \infty.
\end{align*}
Thus, we have that $f(\cdot,y)y \in L^1(\Omega)$.
For the second part of the lemma we consider the projections $y_k = P_k(y)$ and $z_k = P_k(z)$  and use the monotonicity of $f$ as follows
\begin{multline*}
\langle f(\cdot,y) - f(\cdot,z),y - z\rangle_\Omega  = \lim_{k \to \infty}\langle f(\cdot,y) - f(\cdot,z),y_k - z_k\rangle_\Omega
\\
= \lim_{k \to \infty}\int_\Omega (f(\cdot,y) - f(\cdot,z))(y_k - z_k)\dx \ge 0.
\end{multline*}
\end{proof}

\tcr{
Now, we have everything at hand to prove Theorem \ref{T4.1}.
}

{\em Proof of Theorem \ref{T4.1}.} For every integer $k \ge 1$ we define the truncation $f_k(x,s) = f(x,P_k(s))$.
Applying \tcr{monotone operator theory or Schauder's fixed point theorem} we infer the existence of a function $y_k \in H^1(\Omega)$ such that
\begin{equation}
\left\{\begin{array}{l}\displaystyle Ay_k + f_k(\cdot,y_k) =  g \ \mbox{ in } \Omega,\\\partial_{\nu_A}y_k  =  u\ \mbox{ on } \Gamma\ \text{ in } \Omega;\end{array}\right.
\label{E4.7}
\end{equation}
\tcr{
see \cite[Theorem 3.1, Lemma 3.2]{Casas93} or \cite{Casas-Troltzsch2009}.} Testing this equation with $y_k$ and using $f_k(\cdot,s)s \ge 0$, we infer with \eqref{E4.5}
\[
\|y_k\|_{H^1(\Omega)} \le C_1\big(\|g\|_{H^1(\Omega)^*} + \|u\|_{H^{-\frac{1}{2}}(\Gamma)}\big).
\]
Therefore, we take a subsequence, denoted in the same way, such that $y_k \rightharpoonup y$ in $H^1(\Omega)$, $y_k \to y$ in $L^2(\Omega)$, and $y_k(x) \to y(x)$ for almost all $x \in \Omega$.
This implies that $f_k(\cdot,y_k(x)) \to f(\cdot,y(x))$ for almost all $x \in \Omega$.
By {\bfseries (B2)}, there exists $C_{f,1}>0$ such that $|f(x,s)|\le C_{f,1}$ for almost all $x\in \Omega$ and all $|s|\le 1$.
Using the weak formulation, we can derive the following bound
\[\begin{split}
 \int_\Omega|f(\cdot,y_k)|\dx &\le |\Omega|C_{f,1} + \int_\Omega f_k(\cdot,y_k)y_k\dx\\
&= |\Omega|C_{f,1} + \langle g,y_k\rangle_\Omega + \langle u,y_k\rangle_\Gamma - B(y_k,y_k)
\le C_2 < \infty\ \ \forall k \ge 1,
\end{split}\]
 and $\{f_k(\cdot,y_k)y_k\}_{k = 1}^\infty$ is bounded in $L^1(\Omega)$.
Then, from Fatou's lemma we deduce
\[
\int_\Omega|f(\cdot,y)| \dx \le \liminf_{k \to \infty}\int_\Omega|f(\cdot,y_k)|\dx \le C_2 .
\]
Thus, $f(\cdot,y) \in L^1(\Omega)$ holds.
Let us prove that $\{f_k(\cdot,y_k)\}_{k = 1}^\infty$ is equi-integrable.
Given $\varepsilon > 0$ we select $M > 0$ such that $\frac{C_2}{M} < \frac{\varepsilon}{2}$.
Let $C_{f,M}$ be given by {\bfseries (B2)} and take $\delta > 0$ such that $\delta C_{f,M} < \frac{\varepsilon}{2}$.
Then, for every measurable set $E \subset \Omega$ with $|E| < \delta$ and every $k \ge 1$ we have
\[
\int_E|f_k(\cdot,y_k)|\dx \le \frac{1}{M}\int_\Omega f_k(\cdot,y_k)y_k\dx + C_{f,M}|E| \le \frac{C_2}{M} + C_{f,M}\delta < \varepsilon.
\]
Therefore, from Vitali's theorem we deduce that $f_k(\cdot,y_k) \to f(\cdot,y)$ in $L^1(\Omega)$.
Moreover, we have
\[
\langle f_k(\cdot,y_k),z\rangle_\Omega  = \int_\Omega f_k(\cdot,y_k)z\dx = \langle g,z\rangle_\Omega + \langle u,z\rangle_\Gamma - B(y_k,z) \ \  \forall z \in H^1(\Omega),
\]
which implies the boundedness of $\{f_k(\cdot,y_k)\}_{k = 1}^\infty$ in $H^1(\Omega)^*$.
All together yields $f(\cdot,y) \in H^1(\Omega)^*$ and $f_k(\cdot,y_k) \rightharpoonup f(\cdot,y)$ in $H^1(\Omega)^*$.
Further, passing to the limit in the above identity we obtain that $y$ satisfies \eqref{E4.6}.

Let us prove the uniqueness.
If $y_1$ and $y_2$ are solutions of \eqref{E4.1}, subtracting the identities \eqref{E4.6} for $y_2$ and $y_1$ and taking $z = y_2 - y_1$ we infer with \eqref{E4.5} and Lemma \ref{L4.1}-\ref{L4.1-ii}
\[
\Lambda_B\|y_2 - y_1\|^2_{H^1(\Omega)} \le B(y_2 - y_1,y_2 - y_1) + \langle f(\cdot,y_2) - f(\cdot,z_1),y_2 - y_1\rangle_\Omega   = 0.
\]

Finally, we prove the continuous dependence of $y_u$ with respect to $u$.
Let be $\{u_k\}_{k = 1}^\infty$ be a sequence converging
strongly to $u$ in $H^{-\frac{1}{2}}(\Gamma)$.
Taking $u = u_k$ and $z = y_{u_k}$ in \eqref{E4.6} we infer with \eqref{E4.5} and Lemma \ref{L4.1}-\ref{L4.1-ii}
\begin{align*}
\Lambda_B\|y_{u_k}\|^2_{H^1(\Omega)} &\le B(y_{u_k},y_{u_k}) + \langle f(\cdot,y_{u_k}),y_{u_k}\rangle_\Omega \\
& \le C_3\big(\|g\|_{H^1(\Omega)^*} + \|u_k\|_{H^{-\frac{1}{2}}(\Gamma)}\big)\|y_{u_k}\|_{H^1(\Omega)}.
\end{align*}
This implies the boundedness of $\{y_{u_k}\}_{k = 1}^\infty$ in $H^1(\Omega)$ and, consequently, the convergence $y_{u_k} \rightharpoonup y$ in $H^1(\Omega)$ for a subsequence, denoted in the same way. Moreover, using Lemma \ref{L4.1}-\ref{L4.1-i}, the above inequality also leads to the uniform boundedness of the integral $\int_\Omega f(\cdot,y_k)y_k\dx$.
Hence, we can argue as above and deduce the equi-integrability of $\{f(\cdot,y_k)\}_{k = 1}^\infty$ and the convergence $f(\cdot,y_k) \to f(\cdot,y)$ in $L^1(\Omega)$ for a subsequence, again denoted in the same way.
We also have that $f(\cdot,y_k) \rightharpoonup f(\cdot,y)$ in $H^1(\Omega)^*$.
Now, it is easy to pass to the limit in the equations satisfied by $y_{u_k}$ and to deduce that $y = y_u$.
From the uniqueness of the solution of \eqref{E4.6} we get that the whole sequence $\{y_{u_k}\}_{k = 1}^\infty$ converges weakly to $y_u$ in $H^1(\Omega)$.
Finally, the strong convergence follows with \eqref{E4.5} and Lemma \ref{L4.1}-\ref{L4.1-ii}
\begin{align*}
\Lambda_B\|y_{u_k} - y_u\|^2_{H^1(\Omega)} &\le B(y_{u_k} - y_u,y_{u_k} - y_u) + \langle f(\cdot,y_{u_k}) - f(\cdot,y_u),y_{u_k} - y_u\rangle_\Omega\\
& = \langle u_k - u,y_{u_k} - y_u\rangle_\Gamma \to 0\ \text{ as } k \to \infty.
\end{align*}
Now, we prove the convergence of $f(\cdot,y_{u_k}) \to f(\cdot,y_u)$ in $H^1(\Omega)^*$ as follows
\begin{align*}
\|f(\cdot,y_{u_k}) - f(\cdot,y_u)\|_{H^1(\Omega)^*} &= \sup_{\|z\|_{H^1(\Omega)} \le 1}|\langle f(\cdot,y_{u_k}) - f(\cdot,y_u),z\rangle_\Omega|\\
& = \sup_{\|z\|_{H^1(\Omega)} \le 1}|\langle u_k - u,z\rangle_\Gamma - B(y_{u_k} - y_u,z)|\\
&\le C\big(\|u_k - u\|_{H^{-\frac{1}{2}}(\Gamma)} + \|y_{u_k} - y_u\|_{H^1(\Omega)}\big) \to 0  \text{ as } k \to \infty,
\end{align*}
The proof of the convergence $f(\cdot,y_{u_k}) \to f(\cdot,y_u)$ in $L^1(\Omega)$ follows from Vitali's theorem as above taking into account again that $\int_\Omega f(\cdot,y_{u_k})y_{u_k}\dx = \langle f(\cdot,y_{u_k}),y_{u_k}\rangle_\Omega \le C'$ for every $k$.
\endproof

The reader is referred to \cite{Brezis-Browder78} for the study of the Dirichlet problem corresponding to the equation \eqref{E4.1}. See also \cite{BMP2007}.

In the next theorem we establish some $L^q$ estimates for the solution of \eqref{E4.1}.

\begin{theorem}
Let $u\in L^r(\Gamma)$ and $g\in L^s(\Omega)$ with
\[
r \in \left[2\frac{n - 1}{n}, n-1\right), \quad s \in \left[\frac{2n}{n + 2},\frac n2\right)
\]
satisfying
\begin{equation}\label{eq_compat_qq}
 (n-1) \left( \frac1r - \frac1{n-1}\right) = n \left( \frac1s-\frac 2n\right)
\end{equation}
be given.
Let $q$ and $\tilde q$ be defined by
\begin{equation}\label{eq_choice_qq}
 \frac 1q = \frac 1r - \frac 1{n-1}, \quad
 \frac 1{\tilde q} = \frac 1s - \frac 2n.
\end{equation}
Then $y_u \in L^{\tilde q}(\Omega)$ and its trace ${y_u}_{\mid_\Gamma} \in L^q(\Gamma)$ hold.
Moreover, there exists a constant $C = C(r,s)$ independent of $g$ and $u$ such that
\[
\|y_u\|_{L^{\tilde q}(\Omega)} + \|y_u\|_{L^q(\Gamma)} \le C\big(\|g\|_{L^s(\Omega)} + \|u\|_{L^r(\Gamma)}\big).
\]
If  $u\in L^{n - 1}(\Gamma)$ and $g\in L^{\frac{n}{2}}(\Omega)$, then the above estimates are valid for every $q$ and $\tilde q$ smaller than $\infty$.
\label{T4.2}
\end{theorem}

The conditions in \eqref{eq_choice_qq} show the improvements in the integrability, while \eqref{eq_compat_qq}
enforces some compatibility between all these exponents.

\begin{proof}
The proof is similar to the one of Theorem \ref{thm_stateeq_lp_estimates}.
We test the weak formulation with $|y|^{p-2}y$ for suitable $p$ to obtain $H^1$-estimates of $|y|^{p/2}$.
Then the exponents $q$ and $\tilde q$ are derived by applying embedding and trace theorems for $|y|^{p/2} \in H^1(\Omega)$, respectively.
Let us set
 \begin{equation}\label{eq_defp_qq}
 \frac1p =  \frac{n-1}{n-2} \left( \frac1r - \frac1{n-1}\right) = \frac{n}{n-2}  \left( \frac1s-\frac 2n\right),
 \end{equation}
which is well-defined due to \eqref{eq_compat_qq}, and $p \ge 2$ holds due to $r\ge 2\frac{n - 1}{n}$.

 We define $y_k = P_k(y_u)$ for $k \ge 1$.
Then, $|y_k|^{p-2}y_k \in H^1(\Omega) \cap L^\infty(\Omega)$ can be used as test function in the weak formulation, leading to
 \[
B(y,|y_k|^{p-2}y_k) + \int_\Omega f(\cdot,y) |y_k|^{p-2}y_k\dx = \int_\Omega g |y_k|^{p-2}y_k\dx + \int_\Gamma u|y_k|^{p-2}y_k\dx.
 \]
Using \eqref{E4.5} we get
\[\begin{split}
B(y,|y_k|^{p-2}y_k)  &= \int_\Omega \sum_{i,j = 1}^na_{ij}\partial_{x_i}y\partial_{x_j}(|y_k|^{p-2}y_k) + a_0y|y_k|^{p-2}y_k \dx \\
  &\ge \int_\Omega (p-1)\sum_{i,j = 1}^na_{ij}\partial_{x_i}y_k\partial_{x_j}y_k \cdot |y_k|^{p-2} + a_0 |y_k|^p \dx \\
&=\int_\Omega \frac{4(p-1)}{p^2} \sum_{i,j = 1}^na_{ij}\partial_{x_i}(|y_k|^{p/2})\partial_{x_j}(|y_k|^{p/2}) + a_0 (|y_k|^{p/2})^2 \dx\\
& \ge \frac{4(p-1)}{p^2}  B(|y_k|^{p/2},|y_k|^{p/2}) \ge \Lambda_B \frac{4(p-1)}{p^2}\||y_k|^{p/2}\|^2_{H^1(\Omega)},
\end{split}\]
where we used $\frac{4(p-1)}{p^2} \le 1$ for $p\ge2$.
In addition, we have $f(\cdot,y) |y_k|^{p-2}y_k \ge0$.
Hence, we arrive at the inequality
\[
\Lambda_B \frac{4(p-1)}{p^2} \| |y_k|^{p/2} \|_{H^1(\Omega)}^2 \le  \int_\Omega |g|\cdot |y_k|^{p-1}\dx + \int_\Gamma |u|\cdot |y_k|^{p-1}\dx.
\]
Using the continuity of the embedding $H^1(\Omega) \hookrightarrow L^{\frac{2n}{n-2}}(\Omega)$ and of the trace $H^1(\Omega) \hookrightarrow L^{\frac{2n - 2}{n-2}}(\Gamma)$, we infer
\begin{multline}\label{eq39}
 \| y_k \|_{L^{\frac{pn}{n-2}}(\Omega)}^p + \|y_k\|_{L^{\frac{p(n-1)}{n-2}}(\Gamma)}^p \\
 \le \frac{C_1}{\Lambda_B}  \frac{p^2}{4(p-1)} \left( \int_\Omega |g|\cdot |y_k|^{p-1}\dx + \int_\Gamma |u|\cdot |y_k|^{p-1}\dx \right),
\end{multline}
where $C_1 = C_1(n,\Omega)$.
By definition of $p$ in \eqref{eq_defp_qq}, we find
\footnote{The condition here can be written equivalently as $1-\frac1p=(1-\frac1r)\frac{n-1}{n-2}=(1-\frac1s)\frac{n}{n-2}$.}
\[
 \frac1r + \frac{p-1}{p} \frac{n-2}{n-1} =1,
 \quad
 \frac1s + \frac{p-1}{p} \frac{n-2}{n} =1,
\]
and we can apply H\"older and Young inequality to obtain
\begin{equation}\label{eq_C2}
 \| y_k \|_{L^{\frac{pn}{n-2}}(\Omega)}^p + \|y_k\|_{L^{\frac{p(n-1)}{n-2}}(\Gamma)}^p \le C_2\big(\|g\|_{L^s(\Omega)}^p + \|u\|^p_{L^r(\Gamma)}\big).
\end{equation}
The exponents in the above inequality satisfy $\frac{p(n-1)}{n-2}=q$ and $\frac{pn}{n-2}=\tilde q$ by construction.
The claim now follows by taking the limit $k\to \infty$ and Lebesgue's dominated convergence theorem. The last statement of the theorem is an immediate consequence of the first part.
\end{proof}

This theorem is similar to \cite[Theorem 18]{Brezis-Strauss1973},
where $L^{\tilde q}(\Omega)$-$L^s(\Omega)$ estimates are proven for a problem with homogeneous Dirichlet boundary condition.
Note that the above proof cannot be used to derive $L^\infty$-estimates of $y$, see \eqref{eq39}.

\subsection{Analysis of the control problem}
\label{S3.2}

By the usual approach of taking a minimizing sequence, it is immediate to establish the existence of a global minimizer of problem \Pbn with the help of Theorem \ref{T4.1}.
Observe that the weak convergence $u_k \rightharpoonup u$ in $L^2(\Gamma)$ implies the strong convergence $u_k \to u$ in $H^{-\frac{1}{2}}(\Gamma)$.
The goal of this section is to prove that any local (global) minimizer of \Pbn in the $L^2(\Gamma)$ sense is a function of $L^\infty(\Gamma)$.
For this purpose we follow the steps of section \ref{S2.4}.
Given a local minimizer $\bar u$, we take $\rho > 0$ such that $J(\bar u) \le J(u)$ for all $u$ with $\|u-\bar u\|_{L^2(\Gamma)} \le \rho$.
Now, we define the control problems:
\begin{equation}\label{eq_PbMN}
	\tag{P${}_{\mathrm{ell},M}$}
	\min J(u) + \frac12 \|u-\bar u\|_{L^2(\Gamma)}^2
\end{equation}
subject to $\|u-\bar u\|_{L^2(\Gamma)} \le\rho$ and $|u(x)| \le M \text{ f.a.a. } x \in \Gamma$.
\PbMn has at least one solution $u_M$.
Moreover, arguing as in Lemma \ref{L3.1}, we get that $u_M \to \bar u$ in $L^2(\Gamma)$ as $M\to \infty$.
Then, we select $M_0 > 0$ such that $\|u_M - \bar u\|_{L^2(\Gamma)} < \rho$ for every $M > M_0$.
For $M > M_0$, the optimality conditions satisfied by $u_M$ are written as follows
\begin{align}
&\left\{\begin{array}{l}\displaystyle A^*\varphi_M + \frac{\partial f}{\partial y}(\cdot,y_M)\varphi_M =  y_M - y_d \ \mbox{ in } \Omega,\\\partial_{\nu_{A^*}}\varphi_M  =  0\ \mbox{ on } \Gamma,\,\end{array}\right.
\label{E4.8}\\
&\int_\Gamma(\varphi_M + \alpha u_M + u_M -\bar u)(v-u_M)\dx\dt \ge 0 \quad \forall v\in L^2(\Gamma): \ |v| \le M,\label{E4.9}
\end{align}
where $y_M$ is the state associated with $u_M$ and $\varphi_M \in H^1(\Omega) \cap L^\infty(\Omega)$; see \cite[Chapter 4]{Troltzsch2010}.
Observe that $y_M \in H^1(\Omega) \cap L^\infty(\Omega)$ holds due to the assumption {\bfseries (B3)} on $g$ and the fact that $u_M \in L^\infty(\Gamma)$.
As a consequence, we also get with {\bfseries (B3)} that $\varphi_M \in L^\infty(\Omega)$.

Analogously to Theorem \ref{thm_adj_bounded} we have the following result.

\begin{theorem}
Let $\bar u$ be a local minimizer of \Pbn.
Then, $\bar u\in L^\infty(\Gamma)$ holds.
\label{T4.3}
\end{theorem}

The proof of this theorem follows the same arguments used to prove Theorem \ref{thm_adj_bounded} with the obvious changes.
The only difference is that we use the estimates established in Theorem \ref{T4.2} instead of the ones provided in Theorem \ref{thm_stateeq_lp_estimates}.
First we get $L^p(\Omega)$ estimates for the states $y_M$ and with them we derive $L^q(\Gamma)$ estimates for the adjoint state $\varphi_M$.

Once the $L^\infty(\Gamma)$ regularity is proved for any local minimizer of \Pbn, using the differentiability of the mapping $G:L^\infty(\Gamma) \longrightarrow H^1(\Omega) \cap L^\infty(\Omega)$, we can get the first order optimality conditions satisfied by any local minimizer $\bar u$:
\begin{align}
&\left\{\begin{array}{l}\displaystyle A\bar y + f(\cdot,\bar y) = g \ \mbox{ in } \Omega,\\\partial_{\nu_A}\bar y  =  \bar u\ \mbox{ on } \Gamma,\end{array}\right.\label{E4.10}\\
&\left\{\begin{array}{l}\displaystyle A^*\bar\varphi + \frac{\partial f}{\partial y}(\cdot,\bar y)\bar\varphi =  \bar y - y_d \ \mbox{ in } \Omega,\\\partial_{\nu_{A^*}}\bar\varphi  =  0\ \mbox{ on } \Gamma,\end{array}\right.\label{E4.11}\\
&\quad\ {\bar\varphi}_{\mid_\Gamma} + \alpha\bar u = 0.
\label{E4.12}
\end{align}
The reader is referred to \cite[Chapter 4]{Troltzsch2010}.
We have the regularity $\bar y \in H^1(\Omega) \cap C^\mu(\bar\Omega)$ and $\bar\varphi \in H^1(\Omega) \cap C^\mu(\bar\Omega)$ for some $\mu \in (0,1)$; see \cite{AR97,Murthy-Stampacchia72,Nittka2011} for the H\"older regularity.
Moreover, from \eqref{E4.12} the $H^{\frac{1}{2}}(\Gamma) \cap C^\mu(\Gamma)$ regularity of $\bar u$ follows.

\begin{remark}
The arguments used in this section can be applied to the study of the
distributed
control problem
\[
\inf_{u \in L^2(\Omega)} J(u):= \frac{1}{2}\int_\Omega[(y_u - y_d)^2 + \alpha u^2]\dx,
\]
where $y_u$ is the solution of the state equation
\[
\left\{\begin{array}{l}\displaystyle Ay + f(\cdot,y) =  u \ \mbox{ in } \Omega,\\y  =  0\ \mbox{ on } \Gamma.\end{array}\right.
\]
The problem is again well posed in $L^2(\Omega)$ and any local minimizer is a function of $H^1(\Omega) \cap C^\mu(\bar\Omega)$.
To establish the $L^\infty(\Omega)$ boundedness of the control, the arguments relies on the $L^p(\Omega)$ estimates for the states and adjoint states proved in \cite{Brezis-Strauss1973}.
The reader is referred to \cite{CKT2022} for the analysis of this problem with $L^\infty(\Omega)$ controls.
\label{R4.1}
\end{remark}

\end{document}